\documentclass[a4paper,11pt]{article}
\usepackage{amsmath,amsfonts,amsthm,thm-restate}
\usepackage{pstricks}
\usepackage[labelsep=none]{caption}

\newtheorem{theorem}{Theorem}[section]

\newtheorem{lemma}[theorem]{Lemma}
\newtheorem{proposition}[theorem]{Proposition}

\theoremstyle{definition}
\newtheorem{definition}[theorem]{Definition}
\newtheorem{remark}[theorem]{Remark}

\newcommand{\intvl}{\mathrm{I}}				
\newcommand{\crcle}{\mathbb{S}^1}			
\newcommand{\Sphere}{\mathbb{S}^3}			

\DeclareMathOperator{\ms}{MS}				
\DeclareMathOperator{\V}{V}				

\DeclareMathOperator*{\Int}{int}			
\DeclareMathOperator*{\nhd}{\mathcal{N}}		

\begin{document}
\title{On links with locally infinite {K}akimizu complexes}
\author{Jessica E. Banks}
\date{}
\maketitle

\begin{abstract}
We show that the Kakimizu complex of a knot may be locally infinite, answering a question of Przytycki--Schultens. We then prove that if a link $L$ only has connected Seifert surfaces and has a locally infinite Kakimizu complex then $L$ is a satellite of either a torus knot, a cable knot or a connected sum, with winding number 0.
\end{abstract}

\section{Introduction}

The Kakimizu complex $\ms(L)$ of a non-split, oriented link $L$ in $\Sphere$ records the structure of the set of minimal genus Seifert surfaces for $L$. When every minimal genus Seifert surface for $L$ is connected, $\ms(L)$ has the following description, which mirrors the definition of the curve complex of a compact surface.

\begin{definition}[\cite{MR1177053} p225]
$\ms(L)$ is a simplicial complex, the vertices of which are the ambient isotopy classes of minimal genus Seifert surfaces for $L$. Vertices $R_0,\cdots,R_n$ span an $n$--simplex exactly when they can be realised disjointly.
\end{definition}

In \cite{przytycki-2010}, Przytycki and Schultens generalise this definition as follows.

\begin{definition}
Let $M$ be a compact, connected, orientable, irreducible, $\partial$--irreducible 3--manifold. Let $\gamma$ be a union of disjoint, oriented, simple closed curves on $\partial M$ such that $\gamma$ does not separate any component of $\partial M$. Let $\alpha\in H_2(M,\partial M;\mathbb{Z})$ with $\partial\alpha=[\gamma]$.
Call an oriented surface $S$ properly embedded in $M$ a \textit{$(\gamma,\alpha)$--surface} if $[S]=\alpha$ and $\partial S$ is homotopic to $\gamma$.

The flag simplicial complex $\ms(M,\gamma,\alpha)$ is defined as follows. The set $\V(\ms(M,\gamma,\alpha))$ of vertices is defined to be the set of isotopy classes of {$(\gamma,\alpha)$--surfaces} with maximal Euler characteristic $\chi$ in their homology class. Two such surfaces $S,S'$ are joined by an edge if they can be isotoped such that a lift of $M\setminus S'$ to the infinite cyclic cover of $M$ associated to $\alpha$ intersects exactly two lifts of $M\setminus S$.
\end{definition}

\begin{remark}
Using this definition, $\ms(L)$ is ${\ms(\Sphere\setminus\Int\nhd(L),\partial R, [R])}$, where $R$ is any Seifert surface for $L$.
\end{remark}

Viewing $\ms(L)$ in terms of the infinite cyclic cover of its complement in this way has proved especially useful when considering questions about distances in $\ms(L)$. In particular, the following results are proved using this viewpoint.

\begin{theorem}[\cite{MR1177053} Theorem A]
Let $L$ be a non-split link. Then $\ms(L)$ is connected.
\end{theorem}

\begin{theorem}[\cite{MR2531146} Theorem 1.1]\label{diametertheorem}
Let $K$ be a knot in $\Sphere$ that is not a satellite. Then the diameter of $\ms(K)$ is bounded above by $2g(K)(3g(K)-2)+1$, where $g(K)$ denotes the genus of $K$.
\end{theorem}

\begin{theorem}[\cite{przytycki-2010} Theorem 1.1]
If $M,\gamma,\alpha$ are as above, $\ms(M,\gamma,\alpha)$ is contractible.
\end{theorem}

It is known that any knot that is not a satellite has only finitely many minimal genus Seifert surfaces (see, for example, \cite{MR0440528} p329). Contrasting with this and Theorem \ref{diametertheorem}, Kakimizu has shown (\cite{MR1177053} Theorem B) that there are knots $K$ such that $\ms(K)$ has infinite diameter.
Przytycki and Schultens raise the question of whether the complex $\ms(M,\gamma,\alpha)$ can be locally infinite. In Section \ref{examplesection} we give an example that answers this question with the following result.

\begin{theorem}
$\ms$ can be locally infinite even for a knot.
\end{theorem} 

In Section \ref{conversesection} we prove the following condition on the types of links that might have a locally infinite Kakimizu complex, under the additional assumption that all minimal genus Seifert surfaces for the link are connected. Note that such a link cannot be split.

\begin{restatable}{theorem}{satellitethm}\label{satellitetheorem}
Let $L$ be an oriented link such that every minimal genus Seifert surface for $L$ is connected. If $\ms(L)$ is locally infinite then $L$ is a satellite of either a torus knot, a cable knot or a connected sum, with winding number 0.
\end{restatable}

\noindent This, in particular, includes all links with non-zero Alexander polynomial.

\bigskip
I wish to thank Marc Lackenby for helpful conversations, particularly with regard to the proof of Theorem \ref{satellitetheorem}.

\section{A knot with locally infinite $\ms$}\label{examplesection}

\begin{definition}[\cite{przytycki-2010} Section 3]
Let $M$ be a connected 3--manifold, and let $S,S'$ be (possibly disconnected) surfaces properly embedded in $M$ in general position.
$S$ and $S'$ \textit{bound a product region} if the following holds. There is a compact surface $T$, a finite collection $\rho\subseteq \partial T$ of arcs and simple closed curves and an embedding of $T^*=(T\times\intvl)/\!\sim$ into $M$ with the following properties. 
\begin{itemize}
 \item $T\times\{0\}=S\cap T^*$ and $T\times\{1\}=S'\cap T^*$.
 \item $\partial T^*\setminus (T\times\partial\intvl)\subseteq\partial M$.
\end{itemize}
Here $\sim$ collapses $x\times\intvl$ to a point for each $x\in\rho$.
Say $S$ and $S'$ have \textit{simplified intersection} if they do not bound a product region.
\end{definition}

\begin{proposition}[\cite{MR1315011} Proposition 4.8(2)]\label{productregionsprop}
Let $M$ be a $\partial$--irreducible Haken manifold.
Let $S,S'$ be incompressible, $\partial$--incompressible surfaces properly embedded in $M$ in general position. Suppose $S\cap S'\neq\emptyset$, but $S$ can be isotoped to be disjoint from $S'$. Then there is a product region between $S$ and $S'$.
\end{proposition}

\begin{theorem}
Let $K_{\alpha}$ be the twisted Whitehead double of the trefoil shown in Figure \ref{locallyinfinitepic1}.
\begin{figure}[htbp]
\centering

\psset{xunit=.15pt,yunit=.15pt,runit=.15pt}
\begin{pspicture}(700,760)
{
\pscustom[linestyle=none,fillstyle=solid,fillcolor=lightgray]
{
\newpath
\moveto(170,410)
\lineto(200,440)
\lineto(200,680)
\lineto(320,680)
\lineto(350,710)
\lineto(320,740)
\lineto(140,740)
\lineto(140,440)
\lineto(170,410)
\closepath
}
}
{
\pscustom[linestyle=none,fillstyle=solid,fillcolor=lightgray]
{
\newpath
\moveto(50,290)
\lineto(80,320)
\lineto(80,560)
\lineto(320,560)
\lineto(350,590)
\lineto(320,620)
\lineto(20,620)
\lineto(20,320)
\lineto(50,290)
\closepath
}
}
{
\pscustom[linestyle=none,fillstyle=solid,fillcolor=lightgray]
{
\newpath
\moveto(50,290)
\lineto(20,260)
\lineto(20,200)
\lineto(560,200)
\lineto(560,260)
\lineto(530,290)
\lineto(500,260)
\lineto(80,260)
\lineto(50,290)
\closepath
}
}
{
\pscustom[linestyle=none,fillstyle=solid,fillcolor=lightgray]
{
\newpath
\moveto(170,410)
\lineto(140,380)
\lineto(140,20)
\lineto(680,20)
\lineto(680,140)
\lineto(650,170)
\lineto(620,140)
\lineto(200,140)
\lineto(200,380)
\lineto(170,410)
\closepath
}
}
{
\pscustom[linestyle=none,fillstyle=solid,fillcolor=lightgray]
{
\newpath
\moveto(650,170)
\lineto(680,200)
\lineto(680,620)
\lineto(380,620)
\lineto(350,590)
\lineto(380,560)
\lineto(620,560)
\lineto(620,200)
\lineto(650,170)
\closepath
}
}
{
\pscustom[linestyle=none,fillstyle=solid,fillcolor=lightgray]
{
\newpath
\moveto(530,290)
\lineto(560,320)
\lineto(560,740)
\lineto(380,740)
\lineto(350,710)
\lineto(380,680)
\lineto(500,680)
\lineto(500,320)
\lineto(530,290)
\closepath
}
}
{
\pscustom[linestyle=none,fillstyle=solid,fillcolor=gray]
{
\newpath
\moveto(320,140)
\lineto(350,110)
\lineto(380,140)
\lineto(320,140)
\closepath
}
}
{
\pscustom[linestyle=none,fillstyle=solid,fillcolor=gray]
{
\newpath
\moveto(350,110)
\lineto(320,80)
\lineto(350,50)
\lineto(380,80)
\lineto(350,110)
\closepath
}
}
{
\pscustom[linestyle=none,fillstyle=solid,fillcolor=gray]
{
\newpath
\moveto(350,50)
\lineto(320,20)
\lineto(380,20)
\lineto(350,50)
\closepath
}
}
{
\pscustom[linewidth=2,linecolor=black]
{
\newpath
\moveto(200,240)
\lineto(200,220)
\moveto(360,580)
\lineto(380,560)
\lineto(620,560)
\lineto(620,200)
\lineto(680,140)
\lineto(680,20)
\lineto(380,20)
\lineto(360,40)
\moveto(140,240)
\lineto(140,220)
\moveto(500,540)
\lineto(500,320)
\lineto(560,260)
\lineto(560,200)
\lineto(20,200)
\lineto(20,260)
\lineto(40,280)
\moveto(60,300)
\lineto(80,320)
\lineto(80,560)
\lineto(120,560)
\moveto(180,560)
\lineto(160,560)
\moveto(180,620)
\lineto(160,620)
\moveto(560,640)
\lineto(560,740)
\lineto(380,740)
\lineto(320,680)
\lineto(200,680)
\lineto(200,440)
\lineto(180,420)
\moveto(560,580)
\lineto(560,600)
\moveto(540,300)
\lineto(560,320)
\lineto(560,540)
\moveto(500,640)
\lineto(500,680)
\lineto(380,680)
\lineto(360,700)
\moveto(500,580)
\lineto(500,600)
\moveto(520,280)
\lineto(500,260)
\lineto(80,260)
\lineto(20,320)
\lineto(20,620)
\lineto(120,620)
\moveto(340,600)
\lineto(320,620)
\lineto(220,620)
\moveto(660,180)
\lineto(680,200)
\lineto(680,620)
\lineto(380,620)
\lineto(320,560)
\lineto(220,560)
\moveto(160,400)
\lineto(140,380)
\lineto(140,280)
\moveto(340,720)
\lineto(320,740)
\lineto(140,740)
\lineto(140,440)
\lineto(200,380)
\lineto(200,280)
\moveto(140,180)
\lineto(140,20)
\lineto(320,20)
\lineto(380,80)
\lineto(360,100)
\moveto(200,180)
\lineto(200,140)
\lineto(320,140)
\lineto(340,120)
\moveto(640,160)
\lineto(620,140)
\lineto(380,140)
\lineto(320,80)
\lineto(340,60)
}
}
\end{pspicture}
\caption{\label{locallyinfinitepic1}}
\end{figure}
Then $\ms(K_{\alpha})$ is not locally finite.
\end{theorem}
\begin{proof}
Let $R$ be the genus 1 Seifert surface for $K_{\alpha}$ shown in Figure \ref{locallyinfinitepic1} (note that every Whitehead double has such a Seifert surface). We construct an infinite family of genus 1 Seifert surfaces for $K_{\alpha}$ that are disjoint from $R$. 

Let $M=\Sphere\setminus \Int\nhd(K_{\alpha})$. Let $T$ be the torus that bounds the trefoil knot companion of $K_{\alpha}$, such that $K_{\alpha}$ lies in the solid torus bounded by $T$. In addition, let $M_1$ be the part of $M$ outside of $T$ as drawn in Figure \ref{locallyinfinitepic2} (that is, the side away from the knot), and $M_0$ the part on the inside. Let $\mu$ be a meridian of $T\subset\Sphere$. There is a M\"obius band properly embedded in $M_1$, the boundary of which is a longitude $\lambda$ of the solid torus bounded by $T$. Then $\lambda$ and $\mu$ are as shown in Figure \ref{locallyinfinitepic2}. 
\begin{figure}[htbp]
\centering

\psset{xunit=.15pt,yunit=.15pt,runit=.15pt}
\begin{pspicture}(700,700)
{
\pscustom[linewidth=3,linecolor=darkgray,linestyle=dashed,dash=3 6]
{
\newpath
\moveto(360,20.00000523)
\lineto(340,40.00000523)
\lineto(340,60.00000523)
\lineto(360,80.00000523)
}
}
{
\pscustom[linewidth=2,linecolor=black]
{
\newpath
\moveto(540,500.00000523)
\lineto(540,280.00000523)
\lineto(560,260.00000523)
\moveto(380,500.00000523)
\lineto(400,520.00000523)
\lineto(640,520.00000523)
\lineto(640,120.00000523)
\lineto(660,100.00000523)
\lineto(660,40.00000523)
\lineto(160,40.00000523)
\lineto(160,140.00000523)
\moveto(200,540.00000523)
\lineto(300,540.00000523)
\lineto(320,560.00000523)
\moveto(200,380.00000523)
\lineto(180,400.00000523)
\lineto(180,640.00000523)
\lineto(340,640.00000523)
\lineto(360,660.00000523)
\lineto(540,660.00000523)
\lineto(540,560.00000523)
\moveto(160,200.00000523)
\lineto(160,300.00000523)
\lineto(140,320.00000523)
\moveto(500,200.00000523)
\lineto(520,180.00000523)
\lineto(60,180.00000523)
\lineto(60,220.00000523)
\lineto(40,240.00000523)
\lineto(40,540.00000523)
\lineto(140,540.00000523)
\moveto(180,140.00000523)
}
}
{
\pscustom[linewidth=2,linecolor=black,linestyle=dashed,dash=2 4]
{
\newpath
\moveto(500,200.00000523)
\lineto(560,260.00000523)
\moveto(380,500.00000523)
\lineto(320,560.00000523)
\moveto(140,320.00000523)
\lineto(200,380.00000523)
}
}
{
\pscustom[linewidth=3,linecolor=darkgray]
{
\newpath
\moveto(360,80.00000523)
\lineto(380,60.00000523)
\lineto(380,40.00000523)
\lineto(360,20.00000523)
}
}
{
\pscustom[linewidth=4,linecolor=black]
{
\newpath
\moveto(200,140.00000262)
\lineto(200,80.00000262)
\lineto(620,80.00000262)
\lineto(620,500.00000262)
\lineto(200,500.00000262)
\moveto(500,560.00000262)
\lineto(500,620.00000262)
\lineto(200,620.00000262)
\lineto(200,200.00000262)
\moveto(140,500.00000262)
\lineto(80,500.00000262)
\lineto(80,200.00000262)
\lineto(500,200.00000262)
\lineto(500,500.00000262)
\moveto(560,500.00000262)
\lineto(560,140.00000262)
\lineto(20,140.00000262)
\lineto(20,560.00000262)
\lineto(140,560.00000262)
\moveto(140,200.00000262)
\lineto(140,680.00000262)
\lineto(560,680.00000262)
\lineto(560,560.00000262)
\moveto(140,140.00000262)
\lineto(140,20.00000262)
\lineto(680,20.00000262)
\lineto(680,560.00000262)
\lineto(200,560.00000262)
}
}
{
\put(220,360){$\lambda$}
\put(360,100){$\mu$}
}
\end{pspicture}
\caption{\label{locallyinfinitepic2}}
\end{figure}
Let $S_1$ be the annulus properly embedded in $M_1$ that is contained in the boundary of a regular neighbourhood of this M\"obius band in $M_1$. Then $\partial S_1$ is two copies of $\lambda$, with opposite orientations. Let $S_T$ be one of the two annuli into which $T$ is divided by $\partial S_1$.

$R$ is a plumbing of two annuli $S_0$ and $S'_0$ in $M_0$, where $S_0$ is parallel to $S_T$ in $\Sphere\setminus\Int(M_1)$. Isotope $R$ in $M$ so that $R\cap T=S_T$, keeping $\partial R$ fixed. Let $R_0$ be the Seifert surface for $K_{\alpha}$ given by removing $S_T$ from $R$ and replacing it with $S_1$. Then $|R_0\cap T|=2$. In addition, $R_0$ can be made disjoint from $R$.

Express a regular neighbourhood $\nhd(T)$ of $T$ as $\crcle\times\intvl\times\crcle$, where $\crcle\times\{\frac{1}{2}\}\times\{1\}=\mu$ and $\{1\}\times\{\frac{1}{2}\}\times\crcle=\lambda$, and let $S$ be the annulus $\crcle\times\intvl\times\{1\}$. 
Let $\psi\colon S\to S$ be a Dehn twist. Define $\Psi\colon \Sphere\setminus\nhd(K_{\alpha})\to \Sphere\setminus\nhd(K_{\alpha})$ by
\[
\Psi(x)=
\begin{cases}
(\psi(y),z) & \textit{ if }x=(y,z)\in S\times\crcle=\nhd(T)\\
x & \textit{ else.}
\end{cases}
\]
For $n\in\mathbb{Z}$ let $R_n=\Psi^n(R_0)$. Then, for each $n$, $R_n$ is a minimal genus Seifert surface for $K_{\alpha}$ that can be made disjoint from $R$. 
It remains to show that $R_n\neq R$ and $R_n\neq R_m$ for $m\neq n$ when viewed as vertices of $\ms(K_{\alpha})$. 

Fix $n\in\mathbb{Z}$. To show that $R_n\neq R$ we will show that $R_n$ cannot be made disjoint from $T$. In this case we may assume $n=0$. First note that $M$ is $\partial$--irreducible, $R_0$ and $T$ are incompressible, and $T$ is obviously $\partial$--incompressible. $R_0$ is also $\partial$--incompressible as it is orientable, incompressible and not $\partial$--parallel and $\partial M$ is a torus. $M\setminus \Int\nhd(R_0\cup T)$ has three components. One of these is $M_0\setminus\Int\nhd(R_0)$. This is not a product manifold between $R_0$ and $T$ since $R_0$ meets $K_{\alpha}$ in $M_0$ whereas $T$ does not. The other two components lie in $M_1$. One is homeomorphic as a sutured manifold to that shown in Figure \ref{locallyinfinitepic3}, and the other is homeomorphic to its complement. 
\begin{figure}[htbp]
\centering

\psset{xunit=.15pt,yunit=.15pt,runit=.15pt}
\begin{pspicture}(700,700)
{
\pscustom[linestyle=none,fillstyle=solid,fillcolor=lightgray]
{
\newpath
\moveto(60,680)
\lineto(640,680)
\curveto(662.16,680)(680,662.16)(680,640)
\lineto(680,60)
\curveto(680,37.84)(662.16,20)(640,20)
\lineto(60,20)
\curveto(37.84,20)(20,37.84)(20,60)
\lineto(20,640)
\curveto(20,662.16)(37.84,680)(60,680)
\closepath
}
}
{
\pscustom[linewidth=3,linecolor=black,linestyle=dashed,dash=2 4]
{
\newpath
\moveto(680,80.00000262)
\lineto(660,60.00000262)
\lineto(660,40.00000262)
\lineto(140,40.00000262)
\lineto(140,320.00000262)
\lineto(200,380.00000262)
\moveto(380,680.00000262)
\lineto(400,660.00000262)
\lineto(560,660.00000262)
\lineto(560,320.00000262)
\lineto(500,260.00000262)
\moveto(320,680.00000262)
\lineto(380,620.00000262)
\lineto(520,620.00000262)
\lineto(520,340.00000262)
\lineto(500,320.00000262)
\moveto(680,140.00000262)
\lineto(620,80.00000262)
\lineto(180,80.00000262)
\lineto(180,300.00000262)
\lineto(200,320.00000262)
\moveto(320,500.00000262)
\lineto(300,520.00000262)
\lineto(80,520.00000262)
\lineto(80,260.00000262)
\lineto(20,200.00000262)
\lineto(20,260.00000262)
\lineto(40,280.00000262)
\lineto(40,560.00000262)
\lineto(320,560.00000262)
\lineto(380,500.00000262)
}
}
{
\pscustom[linewidth=3,linecolor=black]
{
\newpath
\moveto(380,680.00000262)
\lineto(320,620.00000262)
\lineto(180,620.00000262)
\lineto(180,400.00000262)
\lineto(200,380.00000262)
\moveto(20,260.00000262)
\lineto(80,200.00000262)
\lineto(80,180.00000262)
\lineto(520,180.00000262)
\lineto(520,240.00000262)
\lineto(500,260.00000262)
\moveto(200,320.00000262)
\lineto(140,380.00000262)
\lineto(140,660.00000262)
\lineto(300,660.00000262)
\lineto(320,680.00000262)
\moveto(320,500.00000262)
\lineto(380,560.00000262)
\lineto(660,560.00000262)
\lineto(660,160.00000262)
\lineto(680,140.00000262)
\moveto(500,320.00000262)
\lineto(560,260.00000262)
\lineto(560,140.00000262)
\lineto(40,140.00000262)
\lineto(40,180.00000262)
\lineto(20,200.00000262)
\moveto(680,80.00000262)
\lineto(620,140.00000262)
\lineto(620,520.00000262)
\lineto(400,520.00000262)
\lineto(380,500.00000262)
}
}
{
\pscustom[linewidth=4,linecolor=black]
{
\newpath
\moveto(60,680.00000262)
\lineto(640,680.00000262)
\curveto(662.16,680.00000262)(680,662.16000262)(680,640.00000262)
\lineto(680,60.00000262)
\curveto(680,37.84000262)(662.16,20.00000262)(640,20.00000262)
\lineto(60,20.00000262)
\curveto(37.84,20.00000262)(20,37.84000262)(20,60.00000262)
\lineto(20,640.00000262)
\curveto(20,662.16000262)(37.84,680.00000262)(60,680.00000262)
\closepath
}
}
{
\pscustom[linewidth=4,linecolor=black,fillstyle=solid,fillcolor=white]
{
\newpath
\moveto(240,500)
\lineto(460,500)
\curveto(482.16,500)(500,482.16)(500,460)
\lineto(500,240)
\curveto(500,217.84)(482.16,200)(460,200)
\lineto(240,200)
\curveto(217.84,200)(200,217.84)(200,240)
\lineto(200,460)
\curveto(200,482.16)(217.84,500)(240,500)
\closepath
}
}
\end{pspicture}
\caption{\label{locallyinfinitepic3}}
\end{figure}
Neither of these is a product manifold. By Proposition \ref{productregionsprop}, $R_n$ cannot be isotoped to be disjoint from $T$.

Now fix $m\in\mathbb{Z}$. Again we may assume $n=0$. Let $R'_0$ be a copy of $R_0$, isotoped to be disjoint from $R_0$ (except along its boundary). Then $R'_m=\Psi^m(R'_0)$ is isotopic to $R_m$. Figure \ref{locallyinfinitepic4} shows a cross-section of $\nhd(T)$ in the case $m=2$, 
\begin{figure}[htbp]
\centering
\input{pictexfiles/locallyinfinitepic4}
\caption{\label{locallyinfinitepic4}}
\end{figure}
where $K_{\alpha}$ lies on the inside of $T$ as shown. The components of $M\setminus(R_0\cup R'_m)$ are of five types, as marked. Outside $\nhd(T)$, those marked $M_{0,b}$ and $M_{1,b}$ are each part of the parallel region between $R_0$ and $R'_0$. It is therefore clear that neither of $M_{0,b},M_{1,b}$ is a product region as they each have disconnected intersection with $R_0$. For the same reason, the components of the same type as $M_T$ are not product regions, and neither is $M_{0,a}$. The manifolds $M_{1,a}$ and $M'_{1,a}$ are sutured manifolds and are the same as the components of $M\setminus(R_0\cup T)$ in $M_1$. Hence, again by applying Proposition \ref{productregionsprop}, we see that {$R_0\neq R_m$}.

Thus $\ms(K_{\alpha})$ is locally infinite at $R$.
\end{proof}

\begin{remark}
In \cite{MR1123342}, Kakimizu constructs incompressible Seifert surfaces for a Whitehead double of a knot $K$ using two copies of a Seifert surface for $K$.
Although expressed differently, the above construction is very similar to that used by Kakimizu, with the two Seifert surfaces replaced by the annulus $S'$.
\end{remark}

\section{A restriction on links with locally infinite $\ms$}\label{conversesection}

In this section we prove Theorem \ref{satellitetheorem}. Our proof relies heavily on the work of Wilson in \cite{MR2420023}, to which we refer the reader for definitions not given here. We will also need the following proposition. 

\begin{proposition}[\cite{MR808776} 15.26]\label{knotannulusprop}
Let $K$ be a knot, and let $M=\Sphere\setminus\Int\nhd(K)$. Suppose there is an annulus $S$ properly embedded in $M$ that is not $\partial$--parallel. If neither component of $\partial S$ bounds a disc in $\partial M$ then $K$ is a torus knot, a cable knot, or a connected sum.
\end{proposition}

\begin{definition}
A compact surface $S$ embedded in $\Sphere$ with no closed components is a \textit{spanning surface} for an unoriented link $L$ if $\partial S=L$. We will call $S$ an \textit{unoriented Seifert surface} for $L$ if $S$ is orientable.
\end{definition}

\begin{remark}
An unoriented Seifert surface $R$ for an unoriented link $L$, together with a fixed orientation on $R$, is a Seifert surface for $L$ with the orientation induced by $R$.
\end{remark}

\begin{definition}
Let $S$ be a normal surface in a triangulated 3--manifold.
Its \textit{weight} is the number of times it meets the 1--skeleton of the triangulation. Call $S$ \textit{minimal} if it has minimal weight among normal surfaces isotopic to $S$ by an isotopy fixing $\partial S$.
\end{definition}

\begin{definition}
Let `$+$' denote the usual addition on normal surfaces.
Given normal surfaces $S,S_1,S_2$ with $S=S_1+S_2$, say that $S_1$ and $S_2$ are in \textit{reduced form} if they have been isotoped to minimise $|S_1\cap S_2|$ while maintaining the equation $S=S_1+S_2$.
\end{definition}

In \cite{MR2420023}, Wilson states the following.

\begin{theorem}[\cite{MR2420023} Main Theorem 1.1]\label{wilsonmainthm}
Let $K$ be a non-trivial knot, and let $M=\Sphere\setminus\nhd(K)$. Then there is a finite set $\{R_1,\cdots,R_m\}$ of incompressible Seifert surfaces for $K$ and a finite set $\{S_1,\cdots,S_n\}$ of closed surfaces in $M$ that are not boundary parallel such that any incompressible Seifert surface $R$ is isotopic to a Haken sum $R=R_i+a_1 S_1+\cdots+a_n S_n$, where $a_1,\ldots,a_n$ are non-negative integers.
\end{theorem}

The surfaces $R_1,\ldots,R_m$ that arise from Wilson's proof are spanning surfaces for $K$. However, he does not consider the orientability of these surfaces, which is necessary to conclude, as he does, that they are in fact Seifert surfaces. With some further work it can be shown that it is possible to require these surfaces to be orientable. We will not need this.

It is also worth noting the nature of the isotopy referred to in Theorem \ref{wilsonmainthm}. In his proof, Wilson isotopes the chosen Seifert surface $R$ into normal form based on the following lemma.

\begin{lemma}[\cite{MR2420023} Lemma 3.3]
Let $K$ be a knot, let $M=\Sphere\setminus\nhd(K)$ and let $R$ be an incompressible Seifert surface for $K$ in $M$. Suppose that $M$ is triangulated, and $\partial R$ meets each 2--simplex of the triangulation in at most one normal arc. Then $R$ can be put into normal form by an isotopy fixing $\partial R$.
\end{lemma}

The proof of this lemma gives the stronger conclusion that the isotopy puts the surface into minimal normal form. This is important because minimality is a key hypothesis of \cite{MR744850} Theorem 2.2, which is used in the proof of Theorem \ref{wilsonmainthm}.

Aside from these points, Wilson's proof is actually stronger than the statement of Theorem \ref{wilsonmainthm} suggests. In particular, by following the proof with $M$ the complement of a minimal genus Seifert surface for a link, it gives the following.

\begin{theorem}\label{wilsontheorem}
Let $L$ be an oriented link such that every minimal genus Seifert surface for $L$ is connected. Let $R$ be a minimal genus Seifert surface for $L$, let $M=\Sphere\setminus\Int\nhd(R)$, and fix a set $\rho_1,\cdots,\rho_k$ of core curves of the annuli $\partial M\cap\partial\!\nhd(L)$, one for each link component. There is a triangulation of $M$ such that every Seifert surface $R'$ for $L$ disjoint from $R$ can be put into normal form with $\partial R'=\bigcup_{i=1}^k{\rho_i}$. 

Furthermore, there is a finite set $\{R_1,\cdots,R_m\}$ of surfaces in $M$ with non-empty boundary contained in $\bigcup_{i=1}^k{\rho_i}$, and a finite set $\{S_1,\cdots,S_n\}$ of closed surfaces in $M$, such that all these surfaces are incompressible and in normal form, and the following holds. Any minimal genus Seifert surface $R'$ for $L$ in $M$ with $\partial R'=\bigcup_{i=1}^k{\rho_i}$ and in minimal normal form can be expressed as $a_1 R_1+\cdots+a_m R_m +b_1 S_1+\cdots+b_n S_n$ for some $a_i,b_i\in\mathbb{Z}_{\geq 0}$.
\end{theorem}

If $L$ has more than one component, it is possible that, for a given $j\leq m$, $\partial R_j$ is a strict subset of $\bigcup_{i=1}^k{\rho_i}$. However, only finitely many combinations of $R_1,\cdots,R_m$ will yield the correct boundary. Hence we may assume that $\partial R_j=\bigcup_{i=1}^k{\rho_i}$. Then $\sum_{i=1}^{m}{a_i}=1$.

If $K$ is an oriented knot, any unoriented Seifert surface for $K$ can be oriented to make it a Seifert surface. For a link $L$ with more than one component this might not be the case in general.
The presence of the Seifert surface $R$ for the oriented link $L$ that is disjoint from the spanning surfaces $R_i$ allows us to say more in this case. Suppose that, for some $j$, $R_j$ cannot be oriented to make it a Seifert surface for $L$. Combining it with $R$ then gives a closed, non-orientable surface in $\Sphere$, which is not possible. Hence each $R_i$ is a Seifert surface for $L$.

\satellitethm*
\begin{proof}
Let $R$ be a minimal genus Seifert surface for $L$ such that $\ms(L)$ is locally infinite at $R$. That is, there are infinitely many minimal genus Seifert surfaces for $L$ that can be made disjoint from $R$. Let $M={\Sphere\setminus\Int\nhd(R)}$, and fix a set $\rho_1,\cdots,\rho_k$ of core curves of the annuli $\partial M\cap\partial\!\nhd(L)$, one for each link component. 
Then Theorem \ref{wilsontheorem} applies. In addition, it is clear that none of the $R_i$ is a disc and that, since $R$ is connected, $M$ is irreducible.

By discarding surfaces if necessary, we may ensure that, for any $j\leq n$, the sets $\{R_1,\cdots,R_m\}$ and $\{S_1,\cdots,S_n\}\setminus\{S_j\}$ do not satisfy the conclusions of Theorem \ref{wilsontheorem}. We may also assume that $S_1$ has minimal genus among the $S_i$.
Let $R'$ be a minimal genus Seifert surface in minimal normal form such that $R'=R_1+b_1 S_1 + \cdots + b_n S_n$ with $b_1>0$, and set $T=S_1$. 
Let $R^-=R_1+(b_1-1) S_1 +b_2 S_2+ \cdots + b_n S_n$, so that $R'=R^- + T$, and isotope $R^-$ and $T$ into reduced form. Since the isotopy keeps $\partial R'$ fixed and $T$ is closed, this will leave $\partial R^-$ unchanged.
Then, by \cite{MR744850} Lemma 2.1, no curve of $R^-\cap T$ bounds a disc in either $R^-$ or $T$. Note that although \cite{MR744850} Lemma 2.1 is proved only for closed surfaces, the same proof works in this case because $T$ is closed. 

Suppose that $T$ is a 2--sphere. Then, after the isotopy, it must be disjoint from $R^-$. This contradicts that $R'$ is connected. 
Since there are infinitely many minimal genus Seifert surfaces in minimal normal form in $M$, it follows that $T$ is a torus. 

Let $M_0$ be the component of $M\setminus\Int\nhd(T)$ containing $\partial M$, and $M_1$ the other component.
The orientation that $R'$ inherits from $L$ induces an orientation on each component of $R'\cap M_0$ and hence on each curve of $R^-\cap T$. Let $\rho$ be a curve on $T$ that meets each curve of $R^-\cap T$ once. Because $T$ is disjoint from $R$, the algebraic intersection $\rho.R$ of $\rho$ and $R$ is 0. As $[R']=[R]$ in $\Sphere\setminus\Int\nhd(L)$, this gives that $\rho.R'=0$, and so $\rho.(R^-\cap T)=0$ on $T$. Therefore half the curves of $R^-\cap T$ are oriented in one direction, and half are oriented in the other direction. In particular, $|R^-\cap T|$ is even. Find adjacent curves with opposite orientations, and surger $R^-$ along the subannulus of $T$ between them. Repeating this to remove all curves of $R^-\cap T$ gives a new Seifert surface $R''$ for $L$, together with a closed, possibly disconnected, surface $S''$. Note that $R''\subset M_0$ and $S''$ is orientable.
As $R'$ is minimal genus, $\chi(R')\geq\chi(R'')=\chi(R^-)-\chi(S'')=\chi(R')-\chi(T)-\chi(S'')$, so $\chi(S'')\geq 0$. 
The components of $(R^-\cup T)\setminus(R^-\cap T)$ from which $S''$ is constructed each have boundary, and none of them is a disc. Therefore each of these components is an annulus, and in particular this includes every component of $R^-\cap M_1$.

Let $S$ be one such annulus in $M_1$, and suppose it is parallel to a subannulus $S_T$ of $T$. If there are other curves of $R^-\cap T$ in $S_T$, they must also bound annuli parallel to $T$. Hence we may assume $R^-\cap\Int(S_T)=\emptyset$. At each of the two boundary curves of $S_T$, the cut-and-paste operation that creates $R'$ from $R^-$ and $T$ might go one of two ways (see Figure \ref{satellitepic1}). 
\begin{figure}[htbp]
\centering
\input{pictexfiles/satellitepic1}
\caption{\label{satellitepic1}}
\end{figure}
If both join together $S$ and $S_T$ then this creates a torus component of $R'$, contradicting that $R'$ is connected. If both go the other way, we see that an isotopy of $R^-$ and $T$ could reduce $R^-\cap T$ without changing $R'$, contradicting the choice of $R^-$ and $T$. If only one joins the two annuli, an isotopy along the product region reduces the weight of $R'$, again giving a contradiction.

Thus $S$ is not $\partial$--parallel in $M_1$. 
Note that the part of $\Sphere\setminus \Int\nhd(T)$ containing $L$ is a solid torus $V$. 
Let $K$ be the core curve of $V$.
Since $R\subset V$ and $T$ is incompressible, $L$ is a satellite of $K$ with winding number 0.
Because $S$ is not parallel to $T$, the knot $K$ satisfies the hypotheses of Proposition \ref{knotannulusprop}.
\end{proof}

\bibliography{locallyinfiniterefs}
\bibliographystyle{hplain}

\bigskip
\noindent
Mathematical Institute

\noindent
University of Oxford

\noindent
24--29 St Giles'

\noindent
Oxford OX1 3LB

\noindent
England

\smallskip
\noindent
\textit{jessica.banks[at]lmh.oxon.org}

\end{document}